\documentclass[sn-mathphys,Numbered]{sn-jnl}
\usepackage{graphicx}%
\usepackage{multirow}%
\usepackage{amsmath,amssymb,amsfonts}%
\usepackage{amsthm}%
\usepackage{mathrsfs}%
\usepackage[title]{appendix}%
\usepackage{xcolor}%
\usepackage{textcomp}%
\usepackage{manyfoot}%
\usepackage{booktabs}%
\usepackage{algorithm}%
\usepackage{algorithmicx}%
\usepackage{algpseudocode}%
\usepackage{listings}%
\usepackage{cleveref}
\usepackage{natbib}

\usepackage{cuted}

\let\emptyset\varnothing

\newtheorem{theorem}{Theorem}[section]

\newtheorem{lemma}[theorem]{Lemma}

\theoremstyle{definition}
\newtheorem{definition}[theorem]{Definition}

\newtheorem{remark}[theorem]{Remark}

\numberwithin{equation}{section}
\begin{document}
\title[Strict fixed point problem for contraction-type operators]{Strict fixed point problem, stability results and retraction displacement condition for Picard operators}
\author*[1]{\fnm{Cristina} \sur{Gheorghe}}\email{cristina.gheorghe@ubbcluj.ro}

\author[1,2]{\fnm{Adrian} \sur{Petru\c{s}el}}\email{adrian.petrusel@ubbcluj.ro}

\affil*[1]{ \orgname{Faculty of Mathematics and Computer Science, Babeș-Bolyai
University}, \orgaddress{ \city{Cluj-Napoca}, \postcode{400084}, \country{Romania}}}
\affil[2]{ \orgname{Faculty of Mathematics and Computer Science, Babeș-Bolyai
University}, \orgaddress{ \city{Cluj-Napoca}, \postcode{400084}, \country{Romania}}}



\abstract{
 The aim of this paper is to give strict fixed point principles for multivalued operators $T:X\rightarrow P(X)$ satisfying some contraction conditions of \'Ciri\' c and of \'Ciri\' c-Reich-Rus type. We are interested, under which conditions, the multi-valued operator has a unique strict fixed point and, additionally, when the sequence of its multi-valued iterates $(T^n(x))_{n\in \mathbb{N}}$ converges to this unique strict fixed point. Moreover, some stability properties, such as data dependence on operator perturbation, Ulam-Hyers stability, well-posedness in the sense of Reich and Zaslavski and Ostrowski property of the strict fixed point problem are established. 
}

\maketitle

%%%%%%%%%%%%%%%%%%%%%%%%%%%%%%%%%%%%%%%%%%%%%%%%%%%%%%%%%%%%%%%%%%%%%%%%%%%%%%%%

\section{Introduction}
The main metric  fixed point theorem for multivalued contractions on complete metric spaces was given by Nadler in 1969. Then, one year later, Covitz and Nadler extended the result for the case of a generalized metric space. Later on, there were obtained many other extensions of these results. For example, A. Petru\c sel and G. Petru\c{s}el gave in \cite{AGPetruselSomeVariants} a saturated fixed point principle for multivalued contractions. \\
The purpose of this paper is to formulate similar principles for some elements of the class multivalued Picard operators. More precisely, \'Ciri\' c type operators (see \cite{Reich1972}) and \'Ciri\' c-Reich-Rus type operators (see \cite{Ciric}) will be considered. Our goal is to determine under which assumptions the multi-valued operator has a unique strict fixed point and, moreover, when the sequence of its multivalued iterates $(T^n(x))_{n\in \mathbb{N}}$ is convergent to this unique fixed point. We also study some stability properties, such as data dependence problem, Ulam-Hyers stability, well-posedness in the sense of Reich and Zaslavski and Ostrowski property for the strict fixed point problem.

\section*{Terminology and notations} 

We will use the consecrated notions (see, also \cite{Petrusel}, \cite{Carte55Rus}). In this paper, for a metric space $(X,d)$, $P(X)$ stands for the family of all non-empty subsets of X, $P_{cp}$ represent the family of all non-empty compact subsets of X, and $P_{cl}$ is the family of all non-empty closed subsets of X. Also, for a multivalued operator $T:X\rightarrow P(X)$, the following notations are used:
\begin{itemize}
    \item the fixed point set of T:
    \begin{gather*}
         Fix(T):=\{x\in X: x\in T(x)\}
    \end{gather*}
    \item the strict fixed point set of T:
    \begin{gather*}
        SFix(T):=\{x\in X: \{x\}= T(x)\}
    \end{gather*}
    \item the graphic of T:
    \begin{gather*}
        Graph(T):=\{(x,y)\in X\times X: y\in T(x)\}
    \end{gather*}
\end{itemize}

We also recall, in the context of a metric space,  the definitions of the following functionals:
\begin{itemize}
    \item the gap functional generated by $d$:
\begin{gather*}
    D:P(X)\times P(X)\rightarrow \mathbb{R}_{+
},~D(A,B):=\inf\{d(a,b): a\in A, b\in B\}
\end{gather*}

\item the excess functional of A over B generated by d:
\begin{gather*}
    e:P(X)\times P(X)\rightarrow \mathbb{R}_{+}
\cup\{+\infty\}, ~e(A,B):=\sup\{D(a,B): a\in A \};
\end{gather*}

\item the Pompeiu-Hausdorff functional generated by $d$:
\begin{gather*}
    H:P(X)\times P(X)\rightarrow \mathbb{R}_{+
}\cup\{+\infty\},
~H(A,B):=\max\{e(A,B),e(B,A)\}.
\end{gather*}

\end{itemize}
\begin{definition}\label{def2}
    Let (X,d) be a metric space. The operator $T:X\rightarrow P(X) $ is a multivalued weakly Picard operator (MWP operator) if, for each $x\in X$ and for each $y\in T(x)$, there exists a sequence $(x_n)_{n\in\mathbb{N}}$ in X which satisfies the following conditions:
    \begin{itemize}
        \item[i)] $x_0=x, x_1=y$;
        \item[ii)] $x_{n+1}\in T(x_n)$, for all $n\in \mathbb{N}$;
        \item[iii)] the sequence $(x_n)_{n\in\mathbb{N}}$ is convergent and its limit is a fixed point of T.
    \end{itemize}
\end{definition}
\begin{remark}
    Any sequence satisfying conditions i) and ii) from Definition \ref{def2}, is called a sequence of successive approximations of T.
\end{remark}
\begin{definition}
    We call $T:X\rightarrow P(X)$ a multivalued Picard (MP) operator if the following statements hold true:
    \begin{itemize}
        \item[i)] $SFix(T)=Fix(T)=\{x^*\}$;
        \item[ii)] $T^n(x)\xrightarrow{H}\{x^*\}$ as $n\rightarrow \infty$, for each $x\in X.$ 
    \end{itemize}
\end{definition}
\begin{definition}
     A multivalued Picard operator  $T:X\rightarrow P(X)$ is called $\Psi-$ multivalued Picard operator ($\Psi-MP$ operator) if there exists an increasing mapping $\Psi:\mathbb{R_{+}}\rightarrow \mathbb{R_{+}}$ with $\Psi$ continuous at 0 and $\Psi(0)=0$, such that
     \[d(x,x^*)\leq \Psi(D(x,T(x)), \text{ for all  } x\in X.\]
 \end{definition}
 \begin{remark}
    The above condition is also called the strong retraction-displacement condition for the case of unique (strict) fixed point of $T$. 
\end{remark}

\section{Main results for multivalued \'Ciri\'c type contractions}
\subsection*{Retraction-displacement condition}
%\cite{AGPetruselSomeVariants}
%For results related to retractions of the fixed point set, see \cite{AGPetruselSomeVariants}, \cite{Rus}, \cite{Berinde2016},\cite{SCHRODER2016365} \\
In this section, we will give two main fixed point principles for multivalued retractions satisfying some contraction-type conditions. In a similar manner as in Theorem 5.5 from \cite{AGPetruselSomeVariants}, we can formulate the following first main result:
\begin{theorem}\label{th55}
    Let (X,d) be a complete metric space and $T:X\rightarrow P_{cl}(X)$ be an operator  such that $SFix(T)\neq\emptyset$. Suppose there exist $\alpha, \beta, \gamma\geq 0$ with $\alpha+\beta+\gamma<1$ such that
    \begin{align*}
        H(T(x),T(y))&\leq \alpha d(x,y)+\beta D(x, T(y))+\gamma D(y,T(x)), \text{ for all } x,y\in X.
    \end{align*}
    Then
    \begin{itemize}
        \item[i)] $Fix(T)=SFix(T)=\{x^*\}$;
        \item[ii)] for all $x\in X$, the sequence of sets  $(T^n(x))_{n\in \mathbb{N}}$ converges to $\{x^*\}$ with respect to the Pompeiu-Hausdorff metric $H$;
        \item[iii)]  $d(x,x^*)\leq \dfrac{1+\gamma}{(1-\alpha-\beta)\xi}D(x,T(x))$, for all $x\in X$, where $\xi\in (0,1)$ can be arbitrary chosen.
    \end{itemize}
    \begin{proof}
    i) Let $x^*\in SFix(T)
    $. Suppose there exists $y^*\in SFix(T)$ with $x^*\neq y^*$. Then, we have:
    \begin{align*}
    d(x^*, y^*)&=H(T(x^*), T(y^*))\\& 
    \leq \alpha d(x^*, y^*)+\beta D(y^*,T(x^*))+\gamma D(x^*, T(y^*))\\
    &=\alpha d(x^*, y^*)+\beta d(x^*, y^*)+\gamma d(x^*, y^*).
    \end{align*}
    Therefore, 
    \[(1-\alpha-\beta-\gamma)d(x^*, y^*)\leq 0.\]
    Since $\alpha+\beta+\gamma<1$, it follows that $d(x^*, y^*)=0$, so $x^*= y^*$. The contradiction proves that $SFix(T)=\{x^*\}$.\\
    Let $\tilde{x}\in Fix(T)$ with $x^*\neq\tilde{x}$. Then, we have
    \begin{align*}
    d(\tilde{x},x^*)&=D(\tilde{x},T(x^*))\le H(T(\tilde{x}),T(x^*))\\& 
    \leq \alpha d(\tilde{x},x^*)+\beta D(\tilde{x},T(x^*))+\gamma D(x^*, T(\tilde{x}))\\&
    \le\alpha d(x^*, \tilde{x})+\beta d(\tilde{x}, x^*)+\gamma d(x^*, \tilde{x}).
    \end{align*}
    We get again a contradiction and thus $Fix(T)=\{x^*\}$.
    
    ii) Let $x\in X$ be an arbitrary element and consider the sequence $(T^n(x))_{n\in \mathbb{N}}$. Then, we have:
    \begin{align*}
H(T(x),x^*)=H(T(x),T(x^*))&\le \alpha d(x,x^*)+\beta D(x,T(x^*))+\gamma D(x^*,T(x)) \\&
\le\alpha d(x,x^*)+\beta d(x,x^*)+\gamma H(T(x),x^*).
    \end{align*}
Thus, with the notation $k:=\dfrac{\alpha+\beta}{1-\gamma}\in (0,1)$,  we get:
$$H(T(x), x^*)\leq k d(x,x^*).$$
We prove, by mathematical induction, that the following proposition is true
$$P(n): H(T^n(x),\{x^*\})\leq k^n d(x, x^*),   
   \text{ for all } n\in \mathbb{N}^*. $$
   By the above proof, it is clear that P(1) is true. We will prove that $P(n)\Rightarrow  P(n+1).$\\
   We show first that $H(T(Y),x^*)\leq k H(Y,x^*)$, for all $Y\in P_{cl}(X).$ Indeed, let $v\in T(Y)$ be an arbitrary element. Then, there exists $u\in Y$ such that $v\in T(u)$. We have
   $$d(v, x^*)\leq H(T(u),\{x^*\})\leq kd(u, x^*)\leq k H(Y, \{x^*\}).$$
   Taking $\displaystyle\sup_{v\in T(Y)}$ from the previous inequality, we get that
   \begin{equation}\label{eq}
       H(T(Y),\{x^*\})\leq k H(Y, \{x^*\}).
   \end{equation}
   Therefore, we obtain:
   \begin{align*}
    H(T^{n+1}(x),\{x^*\})&=H(T(T^n(x)), \{x^*\}) \\&\overset{(\ref{eq})}{\leq} kH(T^n(x),\{x^*\})\\&\overset{P(n)}{\leq} k^{n+1} H(T(x), \{x^*\}).   
   \end{align*}
   
   The proof of the mathematical induction is now complete. Hence, \[H(T^n(x), \{x^*\})\leq k^n d(x, x^*)\rightarrow 0  \text{ as } n\rightarrow \infty.\]
   Therefore, $T^n(x)\overset{H}{\rightarrow} \{x^*\}$ as $n\rightarrow \infty.$
   
    iii) Let $\xi\in (0,1)$ be arbitrary and let $x\in X$. Then, there exists $\widetilde{y}\in T(x)$ such that $\xi d(x, \widetilde{y})\leq D(x,T(x))$, or, equivalently, \begin{equation}\label{dr}
            d(x,\widetilde{y})\leq\dfrac{1}{\xi} D(x,T(x)).
        \end{equation}
   
    On the other hand,
    \begin{align*}
    d(x,x^{*})&\leq d(x,\widetilde{y})+d(\widetilde{y},x^{*})
    \\&
    =d(x,\widetilde{y})+D(\widetilde{y},T(x^{*}))\\&
    \leq d(x,\widetilde{y})+H(T(x),T(x^{*}))\\&
    \leq d(x,\widetilde{y})+\alpha d(x,x^{*})+\beta D(x,T(x^*))+\gamma D(x^*, T(x))\\&
    \leq d(x, \widetilde{y})+\alpha d(x, x^*)+\beta d(x, x^*)+\gamma d(x^*,\widetilde{y} ).\label{st}
    \end{align*}
    Hence,
    \begin{equation}\label{st}
        (1-\alpha-\beta)d(x,x^{*})\leq (1+\gamma)d(x,\widetilde{y}).
    \end{equation}
    From \eqref{st} and \eqref{dr}, we get the relation iii).
    \end{proof}
\end{theorem}
For further results related to retractions of a fixed point set, we refer to \cite{AGPetruselSomeVariants}, \cite{Rus}, \cite{articleGraphContr}, \cite{article}, \cite{SCHRODER2016365}.
\subsection*{Well-posedness of the strict fixed point problem}
\begin{definition}
The strict fixed point problem  $T(x)=\{x\}$ is said to be well-posed in the sense of Reich and Zaslavski (see \cite{Reich1972}, \cite{Reich2005ANO}) if $SFix(T)=\{x^{*}\}$ and the following implication holds:
    $$(u_n)_{n\in\mathbb{N}}\subset X \text{ and }D(u_n,T(u_n))\rightarrow 0\implies u_n\rightarrow x^{*} $$
    as  $n\rightarrow \infty$.
\end{definition}

    \begin{theorem}
    Let (X,d) be a complete metric space and $T:X\rightarrow P_{cl}(X)$ be an operator  such that $SFix(T)\neq\emptyset$. Suppose  there exist $\alpha, \beta, \gamma\geq 0$ with $\alpha+\beta+\gamma<1$ such that
    \[H(T(x),T(y))\leq \alpha d(x,y)+\beta D(x, T(y))+\gamma D(y,T(x)), \text{ for all } x,y\in X. \]
     Then, the strict fixed point problem  $T(x)=\{x\}$ is well-posed in the sense of Reich and Zaslavski.
    \end{theorem}

\begin{proof}
    Notice first that, by Theorem 1, $Fix(T)=SFix(T)=\{x^*\}.$\\
     \quad The conclusion follows immediately from the strong retraction-displacement relation iii) in Theorem \ref{th55}, by observing that $$d(u_n, x^*)\leq \dfrac{1+\gamma}{(1-\alpha-\beta)\xi} D(u_n, T(u_n))\rightarrow 0 \text{ as } n\rightarrow \infty.$$
    Alternatively, we have:
    \begin{align*}
        d(u_n,x^*) &\leq D(u_n, T(u_n))+H(T(u_n), T(x^*))\\&
        \nonumber\leq D(u_n, T(u_n))+\alpha d(u_n, x^*)+\beta D(x^*,T(u_n))+\gamma D(u_n, T(x^*))\\&
        \leq D(u_n, T(u_n))+\alpha d(u_n, x^*)+\beta[ d(x^*,u_n)+D(u_n, T(u_n))]+\gamma d(u_n, x^*) 
    \end{align*}
    Then, we get 
        \[d(u_n,x^*)\leq \dfrac{\beta+1}{1-\alpha-\beta-\gamma}D(u_n, T(u_n))\rightarrow 0,\]
         as  $n\rightarrow \infty.$
\end{proof}
\subsection*{Ostrowski stability}
\begin{definition}
    The strict fixed point problem  $T(x)=\{x\}$ has the Ostrowski stability property if 
    $SFix(T)=\{x^{*}\}$ and the following implication holds:
    $$(v_n)_{n\in\mathbb{N}} \subset X, D(v_{n+1}, T(v_n))\rightarrow 0\implies  v_n\rightarrow x^* \mbox{ as } n\rightarrow \infty.$$
\end{definition}
The following auxiliary result will be important in our next proof.
\begin{lemma}
    (Cauchy-Toeplitz lemma, see \cite{CauchyToeplitz}) Let $(a_n)_{n\in \mathbb{N}}$ be a sequence of positive real numbers such that the series $\displaystyle\sum_{n=0}^{\infty} a_n$ is convergent. Let $(b_n)_{n\in\mathbb{N}}$ a sequence of non-negative numbers, that converges to zero. Then:
    $$\lim_{n\rightarrow \infty} \Bigg(\sum_{k=0}^n a_{n-k} b_k\Bigg)=0.$$
\end{lemma}
\begin{theorem}
 Let (X,d) be a complete metric space and $T:X\rightarrow P(X)$ be an operator  such that $SFix(T)\neq\emptyset$. Suppose  there exist $\alpha, \beta, \gamma\geq 0$ with $\alpha+\beta+\gamma<1$ such that
    \[H(T(x),T(y))\leq \alpha d(x,y)+\beta D(x, T(y))+\gamma D(y,T(x)),  \text{ for all } x,y\in X.\]
    The strict fixed point problem  $T(x)=\{x\}$ has the Ostrowski stability property.
\end{theorem}

\begin{proof}
Notice first that, by Theorem \ref{th55}, $Fix(T)=SFix(T)=\{x^*\}.$\\
Then, we have:
    \begin{align*}
        d(v_{n+1}, x^*)&\leq D(v_{n+1}, T(v_n))+H(T(v_n), T(x^*))\\& 
        \nonumber \leq D(v_{n+1}, T(v_n))+\alpha d(v_n, x^*)+\beta D(v_n, T(x^*))+\gamma D(T(v_n), x^*)\\&
        \nonumber = D(v_{n+1}, T(v_n))+(\alpha +\beta) d(v_{n}, x^*)+\gamma D(T(v_n), x^*)\\&
        \nonumber \leq D(v_{n+1}, T(v_n))+(\alpha +\beta) d(v_{n}, x^*)+\gamma(d(v_{n+1}, x^*)+D(v_{n+1}, T(v_n))).
    \end{align*}
    Then, denoting by $k:=\dfrac{\alpha +\beta}{1-\gamma}$, with $0<k<1$, we have
    \begin{align*}
        d(v_{n+1}, x^*)&\leq \frac{1+\gamma}{1-\gamma} D(v_{n+1}, T(v_n))+k d(v_{n},x^*)\\
        &\leq \frac{1+\gamma}{1-\gamma} D(v_{n+1},T(v_n))+k\Big[\frac{1+\gamma}{1-\gamma} D(v_n, T(v_{n-1}))+kd(v_{n-1},x^*)\Big]\\
        &=\frac{1+\gamma}{1-\gamma} D(v_{n+1}, T(v_n))+k\frac{1+\gamma}{1-\gamma} D(v_n, T(v_{n-1}))+k^2 d(v_{n-1},x^*)\\
        &\dots\\
        &\leq \frac{1+\gamma}{1-\gamma}\Big[D(v_{n+1}, T(v_n))+k D(v_{n}, T(v_{n-1})+\cdots + k^n D(v_1, T(v_0))\Big]. 
    \end{align*}
  Since $\displaystyle \sum_{n=0}^\infty k^n $ converges and $\displaystyle\lim_{n\rightarrow\infty} D(v_{n+1}, T(v_n))=0$, from the Cauchy-Toeplitz Lemma, it follows that $d(v_{n+1}, x^*)\rightarrow 0$, so the strict fixed point problem has the Ostrowski stability property. 
    
\end{proof}
\subsection*{Data dependence of the strict fixed point problem}
\begin{definition}\label{def DDE}
    The strict fixed point problem $T(x)=\{x\}$ has the data dependence property if, for each multivalued operator $F:X\rightarrow P(X)$, with the properties:
    \begin{itemize}
        \item[i)] $SFix(F)\neq\emptyset$;
        \item[ii)] there exists $\eta>0$ such that $H(T(x),F(x))\leq \eta$, for all $x\in X;$
    \end{itemize}
    and for each strict fixed point $x^*$ of $T$, there exists a strict fixed point $u^*\in X$ of $F$ such that $d(x^*,u^*)\leq c\eta$, for some $c>0$.
\end{definition}
\begin{theorem}
  Let (X,d) be a complete metric space, $T:X\rightarrow P_{CL}(X)$ such that $SFix(T)\neq \emptyset$ and there exist $\alpha, \beta, \gamma \geq 0$ with $\alpha+\beta+\gamma<1$, such that 
  \[H(T(x),T(y))\leq \alpha d(x,y)+\beta D(x,T(y))+\gamma D(y, T(x)), \text{ for all } x,y\in X.\]
  Then, the strict fixed point problem $T(x)=\{x\}$ has the data dependence property.
  \begin{proof}
      From Theorem \ref{th55}, it is immediately that $
      SFix(T)=Fix(T)=\{x^*\}$ and $$d(x,x^*)\leq \widetilde{L} D(x,T(x)), \text{ for all } x\in X,$$ where $\widetilde{L}=\dfrac{1+\gamma}{(1-\alpha-\beta)\xi}>0$.\\
      Let $F:X\rightarrow P(X)$ a multivalued operator satisfying properties i) and ii) from Definition \ref{def DDE}. Let $u^*\in SFix(F)$. Then, 
      \[d(u^*, x^*)\leq \widetilde{L} D(u^*, T(u^*))= \widetilde{L} H(F(u^*), T(u^*))\leq \widetilde{L} \eta.\]
      \end{proof}
\end{theorem}
\subsection*{Ulam-Hyers stability}
\begin{definition}
    The strict fixed point problem $T(x)=\{x\}$ is Ulam-Hyers  stable if there exists $c>0$ such that for any solution $y^*$ of the relation
    \[D(y,T(y))\leq \varepsilon,\]
    there exists $x^*\in SFix(T)$ satisfying the following relation:
    \[d(y^*, x^*)\leq c\varepsilon.\]
\end{definition}

\begin{theorem}
    Let (X,d) be a complete metric space, $T:X\rightarrow P(X)$ be such that $SFix(T)\neq \emptyset$. Suppose there exist $\alpha, \beta, \gamma \geq 0$ with $\alpha+\beta+\gamma<1$, such that 
  \[H(T(x),T(y))\leq \alpha d(x,y)+\beta D(x,T(y))+\gamma D(y, T(x)), \text{ for all } x,y\in X.\]
  Then, the strict fixed point problem $T(x)=\{x\}$ is Ulam-Hyers stable.
  \begin{proof}
      From Theorem \ref{th55}, we know that $SFix(T)=\{x^*\}$. Let  $y^*$ a solution of the inequation $D(y,T(y))\leq \varepsilon$. Then, 
      \begin{align*}
          d(y^*, x^*)&=D(y^*, T(x^*))\\&\leq D(y^*, T(y^*))+H(T(y^*), T(x^*))\\&
          \leq \varepsilon+\alpha d(x^*,y^*)+\beta D(x^*,T(y^*))+\gamma D(y^*, T(x^*))\\&
          \leq \varepsilon+(\alpha+\beta+\gamma)d(x^*,y^*)+\beta D(y^*, T(y^*))\\&
          \leq (1+\beta)\varepsilon +(\alpha+\beta+\gamma) d(x^*, y^*).
      \end{align*}
    In conclusion, 
    \[d(y^*, x^*)\leq c\varepsilon,\]
    where $c:=\dfrac{1+\beta}{1-\alpha-\beta-\gamma}>0.$
  \end{proof}
  \end{theorem}
  \begin{remark}
      By the above proof, we observe that for each strict fixed point $x^*$ of T, there exists a fixed point $\widetilde{u}$ of F such that $$d(x^*, \widetilde{u})\leq c\eta,  \text{ for some } c>0. $$
      
  \end{remark}

%\begin{theorem}
  %  Let (X,d) be a complete metric space and $T:X\rightarrow P(X)$ a $\Psi$-MWP operator such that $SFix(T)\neq \emptyset$ and the inclusion $T(x)=\{x\}$ is Ulam-Hyers stable. Then, 
  %  \[d(y^*, x^*)\leq \Psi(c\varepsilon).\]
   % \begin{proof}
   %     Since T is a $\Psi-MWP$ operator, we have
    %    \[d(y^*,x^*)\leq \Psi(D(y^*,T(x^*))\leq \Psi(c\varepsilon).\]
   % \end{proof}
%\end{theorem}
\subsection*{Quasi-contraction condition}
\begin{definition}
   Let (X,d) be a complete metric space. An operator $T:X\rightarrow P(X)$ is called a multivalued strong quasi-contraction if $SFix(T)\neq \emptyset$ and there exists $l\in (0,1)$ such that
    \[D(T(x),x^*)\leq l d(x,x^*), \text{ for all } x^*\in SFix(T) \text{ and for all } x\in X.\]
    
\end{definition}
\begin{remark}
    If T is a strong quasi-contraction, then $SFix(T)=\{x^*\}.$
\end{remark}
\begin{proof}
    Let $x^*\in SFix(T)$ and  $y\in SFix(T),$ with $y\neq x^*.$ Then, 
    $$D(T(y), x^*)\leq l d(y, x^*),$$
    or, equivalently $(1-l)d(y, x^*)\leq 0$. This leads to a contradiction. Thus, $SFix(T)$ is a singleton. 
\end{proof}
\begin{theorem}
Let (X,d) be a complete metric space and $T:X\rightarrow P(X)$ be such that $SFix(T)\neq \emptyset$. Suppose  there exist $\alpha, \beta, \gamma \geq 0$ with $\alpha+\beta+\gamma<1$, such that 
  \[H(T(x),T(y))\leq \alpha d(x,y)+\beta D(x,T(y))+\gamma D(y, T(x)), \text{ for all } x,y\in X.\]
    Then, the operator $T$ is a multivalued strong quasi-contraction.
    \begin{proof}
        By Theorem \ref{th55}, we get that $Fix(T)=SFix(T)=\{x^*\}$. \\
        Then, we have
        \begin{align*}
        D(T(x),x^*)&=H(T(x),T(x^*))\\&\leq \alpha d(x,x^*)+\beta D(x,T(x^*))+\gamma D(x^*, T(x)).
        \end{align*}
    Thus, we get that
    \[D(T(x),x^*)\leq l d(x,x^*), \text{ with } l:=\dfrac{\alpha+\beta}{1-\gamma}\in (0,1).\]
        
    \end{proof}
\end{theorem}
\section{Main results for multivalued \'Ciri\' c-Reich-Rus type contractions}

In the following, we will give another fixed point principle for multivalued operators satisfying another contraction-type condition. 
\begin{theorem}\label{th55r2}
    Let (X,d) a complete metric space and let  $T:X\rightarrow P_{cl}(X)$ be an operator with $SFix(T)\neq \emptyset$. Suppose 
    there exists $\alpha, \beta, \gamma \geq 0 $, with $\alpha<1$ such that
    \[H(T(x),T(y))\leq \alpha d(x,y)+\beta D(x, T(x))+\gamma D(y, T(y)), \text{ for all } x,y\in X. \]
    Then:
    \begin{itemize}
        \item[i)] $Fix(T)=SFix(T)=\{x^*\}$;
        \item[ii)] $d(x,x^*)\leq \dfrac{1+\beta}{(1-\alpha)\xi}D(x,T(x))$, for all $x\in X$, where $\xi\in (0,1)$ can be arbitrary chosen.
        \item[iii)]  if, additionally, $\alpha+2\beta<1$, then, for all $x\in X$, the sequence of sets  $(T^n(x))_{n\in \mathbb{N}}$ converges to $\{x^*\}$ with respect to the Pompeiu-Hausdorff metric $H$;
    \end{itemize}
   \begin{proof}
   %The proof is similar to the proof of Theorem \ref{th55}.\\
    i) Let $x^*\in SFix(T)$ and suppose there exists $y^*\in SFix(T)$ such that $x^*\neq y^*$. We get that 
    \begin{align*}
    d(x^*,y^*)&=H(T(x^*), T(y^*))\\&\leq \alpha d(x^*, y^*)+\beta D(x^*, T(x^*))+\gamma D(y^*, T(y^*)).
    \end{align*}
    Hence, $(1-\alpha) d(x^*, y^*)\leq 0$.  \\Since $\alpha<1$, we get the contradiction $x^*=y^*$.Therefore, $SFix(T)=\{x^*\}$.\\
    Let $\widetilde{x}\in Fix(T)$ be a fixed point of $T$ different from $x^*$. Then, 
    \begin{align*}
    d(x,x^*)&=D(\widetilde{x}, T(x^*))\\&\leq H(T(\widetilde{x}), T(x^*))\\&
    \leq \alpha d(\widetilde{x},x^*)+\beta D(\widetilde{x}, T(\widetilde{x}))+\gamma D(x^*, T(x^*))\\
    &=\alpha d(\widetilde{x}, x^*).
    \end{align*}
    Hence, $(1-\alpha)d(\widetilde{x},x^*)\leq 0$ and we get the contradiction which proves that $\widetilde{x}=x^*$. Thus, $Fix(T)=\{x^*\}$.\\
    ii) Let $\xi\in(0,1)$ be an arbitrary value and let $x\in X$. Similarly to the proof of Theorem \ref{th55}, there exists $\widetilde{y}\in T(x)$ such that
    \begin{equation}\label{ec41}
           d(x,\widetilde{y})\leq \frac{1}{\xi}D(x, T(x)).
    \end{equation}
 
    On the other hand, 
    \begin{align*}
    d(x,x^*)&\leq d(x, \widetilde{y})+d(\widetilde{y}, x^*)\\
    &\leq d(x, \widetilde{y})+H(T(x), T(x^*))\\&\leq 
    d(x,\widetilde{y})+\alpha d(x,x^*)+\beta D(x, T(x))+\gamma D(x^*, T(x^*))\\
    &\leq d(x, \widetilde{y})+\alpha d(x, x^*)+\beta d(x, \widetilde{y}).
    \end{align*}
    Hence we get 
    \begin{equation}\label{ec42}
        d(x,x^*)\leq \dfrac{1+\beta}{1-\alpha}d(x, \widetilde{y}).
    \end{equation}
    Thus, from (\ref{ec41}) and (\ref{ec42}), we obtain that
    \[d(x,x^*)\leq \frac{1+\beta}{\xi(1-\alpha)} D(x,T(x)).\]
    iii) Let $x\in X$ be an arbitrary element. Considering the sequence $(T^n(x))_{n\in\mathbb{N}}$, we have
    \begin{align*}
        H(T(x),\{x^*\})&=H(T(x),T(x^*))\\
        &\leq \alpha d(x,x^*)+\beta D(x,T(x))+\gamma D(x^*,T(x^*))\\&
        \leq \alpha d(x,x^*)+\beta [d(x,x^*)+D(x^*,T(x))]\\
        &\leq(\alpha+\beta) d(x,x^*)+\beta H(T(x), \{x^*\}). 
    \end{align*}
    In consequence, $H(T(x),\{x^*\})\leq \dfrac{\alpha+\beta}{1-\beta} d(x,x^*)$.\\ By denoting $k:=\dfrac{\alpha+\beta}{1-\beta}$, we observe that $0<k<1$ and we can easily prove by mathematical induction that 
    \[P(n): H(T^n(x),\{x^*\})\leq k^n d(x,x^*)\]
    holds true for every $n\in \mathbb{N}^*$. Therefore, $(T^n(x))_{n\in \mathbb{N}}$ converges to $\{x^*\}$ with respect to the Pompeiu-Hausdorff metric $H$.
   \end{proof}
   \begin{theorem}
       Let (X,d) be a complete metric space and $T:X\rightarrow P_{cl}(X)$ be such that $SFix(T)\neq \emptyset$. Suppose there exist $\alpha, \beta, \gamma \geq 0$ with $\alpha<1$ such that
  \[H(T(x),T(y))\leq \alpha d(x,y)+\beta D(x,T(x))+\gamma D(y, T(y)), \text{ for all } x,y\in X.\]
  Then, the strict fixed point problem $T(x)=\{x\}$ has the following properties:
  \begin{itemize}
      \item[i)] is well-posed in the sense of Reich and Zaslavski;
      \item[ii)] has the data dependence property;
      \item[iii)] is Ulam-Hyers stable;
      \item[iv)] has the Ostrowski stability property;
      \item[v)] the operator $T$ is a multivalued strong quasi-contraction.
      
  \end{itemize}
  \begin{proof}
  By Theorem \ref{th55r2}, we have $Fix(T)=SFix(T)=\{x^*\}.$\\
      i) From the strong retraction-displacement condition iii), from Theorem \ref{th55r2}, we get
      \[d(u_n, x^*)\leq \frac{1+\beta}{(1-\alpha)\xi} D(u_n, T(u_n))\rightarrow 0 \text{ as } n\rightarrow\infty\]
      ii) From Theorem \ref{th55r2}, we get $Fix(T)=SFix(T)=\{x^*\} $ and 
      \[d(x,x^*)\leq \widetilde{K} D(x,T(x)), \text{ for all } x\in X, \]
      with $\widetilde{K}:=\dfrac{1+\beta}{(1-\alpha)\xi}>0.$\\
      Let $F:X\rightarrow P(X)$ a multivalued operator such that:
      \begin{itemize}
          \item[a)] $SFix(F)\neq \emptyset$;
          \item[b)] there exists $\eta>0$ such that $H(T(x), F(x))\leq \eta,$ for all $x\in X.$
      \end{itemize}
      Let $u^*\in SFix(F).$ Then,
      \[d(u^*, x^*)\leq \widetilde{K} D(u^*, T(u^*))\leq \widetilde{K} H(F(u^*), T(u^*))\leq \widetilde{K}\eta.\]
      iii) Let $y^*\in X$ a solution of the inequation $D(y,T(y))\leq \varepsilon.$ Then,
      \begin{align*}
          d(y^*,x^*)&=D(y^*,T(x^*))\leq D(y^*, T(y^*))+H(T(y^*), T(x^*))\\&
          \leq \varepsilon +\alpha d(x^*, y^*)+\beta D(y^*,T(y^*))+\gamma D(x^*, T(x^*))\\&
          \leq \varepsilon +(\alpha+\beta) d(x^*, y^*)+\beta D(x^*, T(y^*))\\&
          \leq (1+\beta)\varepsilon +(\alpha+\beta) d(x^*, y^*).
      \end{align*}
      Hence, $$d(y^*,x^*)\leq c\varepsilon, $$
      with $c:=\dfrac{1+\beta}{1-\alpha-\beta}>0.$\\
      iv) Let $(v_n)\subset X$. We get
      \begin{align*}
      d(v_{n+1}, x^*)&\leq D(v_{n+1}, T(v_n))+H(T(v_n), T(x^*))\\&
      \leq D(v_{n+1}, T(v_n))+\alpha d(v_n, x^*)+\beta D(v_n, T(v_n))+\gamma D(T(x^*), x^*)\\&
      \leq D(v_{n+1}, T(v_n))+\alpha d(v_n,x^*)+\beta d(v_n, x^*)+\beta d( x^*,v_{n+1})+\\&+\beta D(v_{n+1}, T(v_n)).
      \end{align*}
      Thus, by denoting $k:=\dfrac{\alpha+\beta}{1-\beta}$, with $0<k<1$, we obtain: 
      \begin{align*}
      d(v_{n+1},x^*)&\leq \frac{1+\beta}{1-\beta} D(v_{n+1}, T(v_n))+k d(v_n, x^*).\\
      &\leq \frac{1+\beta}{1-\beta} D(v_{n+1}, T(v_n))+ k\Big[\frac{1+\beta}{1-\beta} D(v_n, T(v_{n-1})+k d(v_{n-1}, x^*)\Big]\\
      &\leq \frac{1+\beta}{1-\beta} \Big[ D(v_{n+1}, T(v_n))+ k D(v_n, T(v_{n-1}))+\\&+ k^2 D(v_{n-1}, T(v_{n-2}))+k^3 d(v_{n-2}, x^*)\Big]\\
      & \dots\\
      & \leq \frac{1+\beta}{1-\beta} \Big[ D(v_{n+1}, T(v_n))+ k D(v_n, T(v_{n-1}))+k^2 D(v_{n-1}, T(v_{n-2}))+\cdots \\& \cdots + k^n D(v_1, T(v_0))\Big]
      \end{align*}
      By the Cauchy-Toeplitz Lemma, we get the conclusion.\\
       %Providing $\displaystyle\lim_{n\rightarrow \infty} D(v_{n+1}, T(v_n))=0$, from Cauchy-Toeplitz Lemma, we obtain $v_n\rightarrow 0$ as $n\rightarrow \infty.$\\
       v) Let $x^*\in SFix(T).$ We have:
       \begin{align*}
          H(T(x),\{x^*\})&\leq \alpha d(x,x^*)+\beta D(x, T(x))+\gamma D(x^*, T(x^*)) \\&
           \leq (\alpha+\beta) d(x,x^*)+\beta D(x^*, T(x))
       \end{align*}
     Hence, 
     \[D(T(x), x^*)\leq l d(x,x^*), \]
     where $l:=\dfrac{\alpha+\beta}{1-\beta}d(x,x^*)\in (0,1).$
  \end{proof}
   \end{theorem}
\end{theorem}
\begin{remark}
    In a similar manner as the two cases from Theorem \ref{th55} and Theorem \ref{th55r2}, we can obtain similar results for the operator $T:X\rightarrow P(X),$ with $SFix(T)\neq \emptyset$ such that there exists $\alpha, \beta, \gamma\in \mathbb{R_{+}},  \alpha+2\gamma<1$ such that
    \begin{align*}
        H(T(x),T(y))&\leq \alpha d(x,y)+ \beta [D(x, T(x))+D(y, T(y))]+\gamma [D(x, T(y))+D(y, T(x))].
    \end{align*}
\end{remark}

\medskip

{\bf Open Question.} By the above results, an important issue rises: under which concrete and verifiable conditions on the multi-valued operator $T:X\to P(X)$ its strict fixed point set $SFix(T)$ is nonempty ? 

Here below we will list such sufficient conditions from the literature.

\begin{theorem}(see J. Jachymski \cite{JACHYMSKI2011169})
Let $(X,d)$ be a complete metric space and $T:X\to P(X)$ be a multi-valued operator. Suppose that the following conditions are satisfied:

\quad (i) $T(T(x))\subset T(x)$, for each $x\in X$;

\quad (ii) for each $x\in X$ and each $\varepsilon>0$ there exists $y\in T(x)$ such that 
$diam(T(y))<\varepsilon$.

Then $SFix(T)\neq\emptyset$.
    \end{theorem}

    \begin{theorem}(see J.-P. Aubin, J. Siegel \cite{AS})
        Let $(X,d)$ be a complete metric space and $T:X\to P(X)$ be a multi-valued operator. Suppose that:

        \quad (i) $T$ is lower semi-continuous;

        \quad (ii) there exists a function $\psi:X\to \mathbb{R}_+$ such that
        \begin{equation}
            \mbox{ for every } x\in X \mbox{ and every } y\in T(x) \mbox{ we have } d(x,y)\le \psi(x)-\psi(y). 
        \end{equation}
        Then $SFix(T)\neq\emptyset$.
    \end{theorem}

\begin{theorem}(see J.-P. Aubin, J. Siegel \cite{AS})
        Let $(X,d)$ be a complete metric space and $T:X\to P(X)$ be a multi-valued operator. Suppose that there exists a lower semi-continuous function $\psi:X\to \mathbb{R}_+$ such

        \begin{equation}
            \mbox{ for every } x\in X \mbox{ and every } y\in T(x) \mbox{ we have } d(x,y)\le \psi(x)-\psi(y). 
        \end{equation}
        Then $SFix(T)\neq\emptyset$.
    \end{theorem}
\begin{theorem}(see J. Tiammee, S. Suantai \cite{TIAMMEESUANTAI2017}  )
    Let (X,d) be a complete metric space and let $T:X\rightarrow P_b(X)$ be a multivalued operator satisfying the following properties:
    \begin{itemize}
        \item[1)] there exist two  continuous and non-decreasing functions $\psi, \theta:[0,\infty)\rightarrow [0,\infty)$ with $\psi(t)=0 $ iff $t=0$ and $\theta(t)=0$ iff $t=0$ satisfying the following property:
        \[\psi(H(T(x), T(y))\leq \psi(d(x,y))-\theta(d(x,y)), \text{ for all } x,y\in X\]
        \item[2)] there exists $\alpha\in (0,1)$ such that for all $u\in T(x)$, there exists $v\in T(y)$ such that
        \[d(u,v)\leq \alpha d(x,y).\]
    \end{itemize}
    Then $SFix(T)\neq \emptyset$.
\end{theorem}
\begin{theorem}(see R. Espínola, M. Hosseini, K. Nourouzi \cite{ESPINOLA2015})
Let X be a Banach space with characteristic of convexity $\epsilon_0(X)\leq 1.$ Let C be a nonempty, weakly compact and convex subset of X and let $T:X\rightarrow P(X)$ be a non-expansive multivalued operator, i.e.,
$$H(T(x),T(y))\le \|x-y\|, \mbox{ for every } x,y\in X.$$ Suppose that 
\[\lim _{n \rightarrow \infty} \sup_{z\in T(x_n)}\|x_n-z\|=0. \]
Then $SFix(T)\neq \emptyset$.
   
\end{theorem}

\end{document}